\documentclass[10pt]{amsart}
\usepackage[margin=1.2in]{geometry}
\usepackage[notcite,notref]{showkeys}
\usepackage{xcolor}
\newtheorem{theorem}{Theorem}[section]
\newtheorem{lemma}[theorem]{Lemma}
\newtheorem{proposition}{Proposition}[section]

\theoremstyle{definition}

\theoremstyle{remark}
\newtheorem{remark}[theorem]{Remark}
\usepackage{cite}
\numberwithin{equation}{section}

\begin{document}
\title[Multi-species quantum BGK model]{BGK model of the multi-species Uehling-Uhlenbeck equation}

\author{Gi-Chan Bae}
\address{Department of mathematics, Sungkyunkwan University, Suwon 16419, Republic of Korea }
\email{gcbae02@skku.edu}

\author{Christian Klingenberg}
\address{Department of mathematics, W\"urzburg University, Emil Fischer Str. 40, 97074 W\"urzburg, GERMANY}
\email{klingen@mathematik.uni-wuerzburg.de}

\author{Marlies Pirner}
\address{Department of mathematics, Vienna University, Oskar-Morgenstern-Platz 1, 1090 Vienna, Austria}
\email{marlies.pirner@mathematik.uni-wuerzburg.de}

\author{Seok-Bae Yun}
\address{Department of mathematics, Sungkyunkwan University, Suwon 16419, Republic of Korea }
\email{sbyun01@skku.edu}




\begin{abstract}
We propose a BGK model of the quantum Boltzmann equation for gas mixtures. We also provide a sufficient condition that guarantees the existence of equilibrium coefficients so that the model shares the same conservation laws and $H$-theorem with the quantum Boltzmann equation. Unlike the classical BGK for gas mixtures, the equilibrium coefficinets of the local equilibiriums for quantum multi-species gases are defined through highly nonlinear relations that are not explicitly solvable. We verify in a unified way that such nonlinear relations uniquely determine the equilibrium coefficients under the condition, leading to the well-definedness of our model.
\end{abstract}

\maketitle
\tableofcontents
\section{Introduction}
\subsection{Quantum Boltzmann equation for gas mixture}
The quantum modification of the celebrated Boltzmann equation was made in \cite{uehling1933transport,uehling1934transport} to incorporate the quantum effect that cannot be neglected for light molecules (such as Helium) at low temperature. Quantum Boltzmann equation is now fruitfully employed not just for low temperature gases, but in various circumstances such as the study of carrior mobility in various electronic devices. 
When the gas is composed of several different types of molecules (gas mixture), the quantum Boltzmann equation takes the form (For simplicity, we restrict ourselves to two species case):
\begin{align}\label{MQBE}
	\begin{aligned}
	\partial_t f_1+\frac{p}{m_1}\cdot \nabla_xf_1=Q_{11}(f_1,f_1)+Q_{12}(f_1,f_2), \cr
	\partial_t f_2+\frac{p}{m_2}\cdot \nabla_xf_2=Q_{22}(f_2,f_2)+Q_{21}(f_2,f_1).
\end{aligned}
\end{align}
The momentum distribution function $f_i(x,p,t)$ denotes the number density at the phase point $(x,p)\in \Omega_x\times \mathbb{R}^3_p$ at time $t$. The collision operator $Q_{ij}$ $(i,j=1,2)$ takes the following form:\newline

$\bullet$ Fermion-Fermion $(-)$, Boson-Boson $(+)$.
\begin{align*}
Q_{ij}(f_i,f_j)=\int_{\mathbb{R}^3}\int_{\mathbb{S}^2}B_{ij}\left(\bigg|\frac{p}{m_i}-\frac{p_*}{m_j}\bigg|,w\right)
\{f_i'f_{j,*}'(1\pm f_i)(1\pm f_{j,*})-f_if_{j,*}(1\pm f_i')(1\pm f_{j,*}')\} dwdp_*,
\end{align*}

$\bullet$ Fermion ($f_1$)-Boson ($f_2$) interaction:
\begin{align*}
Q_{ij}(f_i,f_j)&=\int_{\mathbb{R}^3}\int_{\mathbb{S}^2}B_{ij}\left(\bigg|\frac{p}{m_i}-\frac{p_*}{m_j}\bigg|,w\right)
\{f_i'f_{j,*}'(1+\tau(i) f_i)(1+\tau(j)f_{j,*})\cr
&\hspace{4cm}-f_if_{j,*}(1+\tau(i) f_i')(1+ \tau(j)f_{j,*}')\} dwdp_*,
\end{align*}
\hspace{3mm} where $\tau(1)=-1$, $\tau(2)=1$. We assume $B_{12}\left(\cdot,w\right)=B_{21}\left(\cdot,w\right)$ for both cases, and
we used the abbreviated notation:
\begin{align*}
	f_i=f_i(x,p,t), ~f_{i,*}=f_i(x,p_*,t), ~f_i^{\prime}=f_i(x,p',t),~f^{\prime}_{i,*}=f_i(x,p_*',t), ~ i=1,2.
\end{align*}
The relation between the pre-collisonal momenta $(p,p_*)$, and the post-collisional momenta $(p',p'_*)$ in $Q_{ij} ~(i,j=1,2)$ can be derived from the local conservation laws:
\begin{align}\label{mom_after_coll}
\begin{split}
p'+p'_*&=p+p_*, \cr
\frac{|p'|^2}{2m_i}+\frac{|p_*'|^2}{2m_j}&=\frac{|p|^2}{2m_i}+\frac{|p_*|^2}{2m_j},
\end{split}
\end{align}
in the following explicit forms:
\begin{align*}
	p'&= p-\frac{2m_im_j}{m_i+m_j}w\left[\left(\frac{p}{m_i}-\frac{p_*}{m_j}\right)\cdot w\right], \cr
	p_*'&= p_*+\frac{2m_im_j}{m_i+m_j}w\left[\left(\frac{p}{m_i}-\frac{p_*}{m_j}\right)\cdot w\right].
\end{align*}

The collision operator has 5 collision invariants: $1,p,|p|^2$~ $(k=1,2)$:
\begin{align}
	\begin{split}\label{cancellation}
		&\int_{\mathbb{R}^3}Q_{kk}(f_k,f_k)dp=0,\hspace{0.2cm}\int_{\mathbb{R}^3}Q_{12}(f_1,f_2) dp=\int_{\mathbb{R}^3}Q_{21}(f_2,f_1) dp = 0,\cr
		&\int_{\mathbb{R}^3}Q_{kk}(f_k,f_k)pdp=0,\hspace{0.2cm}\int_{\mathbb{R}^3}\left\{Q_{12}(f_1,f_2)+Q_{21}(f_2,f_1)\right\}p dp = 0 ,\cr
		&\int_{\mathbb{R}^3}Q_{kk}(f_k,f_k)\frac{|p|^2}{2m_k}dp=0,\hspace{0.2cm}\int_{\mathbb{R}^3}\left\{Q_{12}(f_1,f_2)\frac{|p|^2}{2m_1}+Q_{21}(f_2,f_1)\frac{|p|^2}{2m_2}\right\} dp = 0,
	\end{split}
\end{align}
which leads to the conservation of total mass, momentum and energy:
\begin{align}
\begin{split}\label{conservation}
&\frac{d}{dt}\int_{\mathbb{T}^3\times \mathbb{R}^3}f_1dxdp=0, \quad \frac{d}{dt}\int_{\mathbb{T}^3\times \mathbb{R}^3}f_2dxdp=0, \cr
&\frac{d}{dt}\left(\int_{\mathbb{T}^3\times \mathbb{R}^3}f_1pdxdp+\int_{\mathbb{T}^3\times \mathbb{R}^3}f_2pdxdp\right)=0, \cr
&\frac{d}{dt}\left(\int_{\mathbb{T}^3\times \mathbb{R}^3}f_1\frac{|p|^2}{2m_1}dxdp+\int_{\mathbb{T}^3\times \mathbb{R}^3}f_2\frac{|p|^2}{2m_2}dxdp\right)=0.
\end{split}
\end{align}
Upon defining the velocity distribution function $\bar{f}_i(x,v,t)$ by the following relation with respect to the momentum distribution $f_i(x,p,t)$:
\[
\bar{f}_i(x,v,t)= m_i^3f_i(x,p,t),\quad \left(v=\frac{p}{m_i}\right)
\]
we can recover the usual conservation laws as in \cite{MR2370359,MR3579578,garzo1989kinetic}. (See Appendix). The collision operator $Q_{ij}$ $(i,j\in\{1,2\})$ also satisfies the following entropy dissipation property:
\begin{align}\label{entropy} 
	\begin{split}
		\int_{\mathbb{R}^3}\ln \frac{f_1}{1+\tau(1) f_1} Q_{11}(f_1,f_1) dp \leq 0&,\quad \int_{\mathbb{R}^3}\ln \frac{f_2}{1+\tau(2) f_2} Q_{22}(f_2,f_2) dp \leq 0,\cr
		\int_{\mathbb{R}^3}\ln \frac{f_1}{1+\tau(1) f_1} Q_{12}(f_1,f_2) dp&+\int_{\mathbb{R}^3}\ln \frac{f_2}{1+\tau(2) f_2}Q_{21}(f_2,f_1) dp\leq 0.
	\end{split}
\end{align}
where $\tau(i)=-1 $ when $f_i$ denotes distribution of fermion and $\tau(i)=+1 $ when $f_i$ denotes distribution of boson.

Such dissipation implies the celebrated $H$-theorem for quantum mixture:

$\bullet$ Fermion-Fermion $(-)$, Boson-Boson $(+)$:
\begin{align*}
	\frac{d}{dt}H(f_1,f_2)&=\int_{\mathbb{R}^3}\ln \frac{f_1}{1\pm f_1} Q_{11}(f_1,f_1) dp+\int_{\mathbb{R}^3}\ln \frac{f_2}{1\pm f_2} Q_{22}(f_2,f_2) dp\cr
	&+\int_{\mathbb{R}^3}\ln \frac{f_1}{1\pm f_1} Q_{12}(f_1,f_2) dp+\int_{\mathbb{R}^3}\ln \frac{f_2}{1\pm f_2} Q_{21}(f_2,f_1) dp\leq0,
\end{align*}

$\bullet$ Fermion ($f_1$)-Boson ($f_2$):
\begin{align*}
\frac{d}{dt}H(f_1,f_2)&=\int_{\mathbb{R}^3}\ln \frac{f_1}{1-f_1} Q_{11}(f_1,f_1) dp+\int_{\mathbb{R}^3}\ln \frac{f_2}{1+f_2} Q_{22}(f_2,f_2) dp\cr
&+\int_{\mathbb{R}^3}\ln \frac{f_1}{1-f_1} Q_{12}(f_1,f_2) dp+\int_{\mathbb{R}^3}\ln \frac{f_2}{1+f_2} Q_{21}(f_2,f_1) dp\leq0,
\end{align*}

where $H(f_1,f_2)$ denotes  the $H$-functional:

$\bullet$ Fermion-Fermion interaction:
\[
H(f_1,f_2)=\int_{\mathbb{R}^3} f_1\ln f_1+(1-f_1)\ln (1-f_1)dp+\int_{\mathbb{R}^3} f_2\ln f_2+(1-f_2)\ln (1-f_2)dp.
\]

$\bullet$  Boson-Boson interaction:
\[
H(f_1,f_2)=\int_{\mathbb{R}^3} f_1\ln f_1-(1+f_1)\ln (1+f_1)dp+\int_{\mathbb{R}^3} f_2\ln f_2-(1+f_2)\ln (1+f_2)dp.
\]

$\bullet$ Fermion ($f_1$)-Boson ($f_2$) interaction:
\[
H_{FB}(f_1,f_2)=\int_{\mathbb{R}^3} f_1\ln f_1+(1-f_1)\ln (1-f_1)dp+\int_{\mathbb{R}^3} f_2\ln f_2-(1+f_2)\ln (1+f_2)dp.
\]
The r.h.s of (\ref{MQBE}) vanishes if and only if $f_1$ and $f_2$ are quantum equilibrium: \newline
\begin{itemize}
\item Fermion-Fermion $(+)$, Boson-Boson interaction $(-)$:
\begin{align*}
f_1(x,p,t)=\frac{1}{e^{m_1a(x,t)\big|\frac{p}{m_1}-b(x,t)\big|^2+c_1(x,t)}\pm1},\quad
f_2(x,p,t)=\frac{1}{e^{m_2a(x,t)\big|\frac{p}{m_2}-b(x,t)\big|^2+c_2(x,t)}\pm1}.
\end{align*}

\item Fermion ($f_1$)-Boson ($f_2$) interaction
\begin{align*}
f_1(x,p,t)=\frac{1}{e^{m_1a(x,t)\big|\frac{p}{m_1}-b(x,t)\big|^2+c_1(x,t)}+1},\quad
f_2(x,p,t)=\frac{1}{e^{m_2a(x,t)\big|\frac{p}{m_2}-b(x,t)\big|^2+c_2(x,t)}-1}.
\end{align*}
\end{itemize}

\subsection{Quantum BGK model for gas mixture}
 In this paper, we propose a BGK type relaxation model of (\ref{MQBE}) :
\begin{align}\label{MQBGK}
\begin{split}
	\partial_t f_1+\frac{p}{m_1}\cdot \nabla_xf_1=\mathcal{R}_{11}+\mathcal{R}_{12}, \cr
	\partial_t f_2+\frac{p}{m_2}\cdot \nabla_xf_2=\mathcal{R}_{21}+\mathcal{R}_{22},
\end{split}
\end{align}
where $\mathcal{R}_{ij}$ denotes the relaxation operator for the interactions of $i$th and $j$th component. More explicitly, they are defined as follows:
\begin{itemize}
\item  Fermion-Fermion interaction ($i\neq j$):
 \[
 \mathcal{R}_{ii}=\mathcal{F}_{ii}-f_{i},\quad \mathcal{R}_{ij}=\mathcal{F}_{ij}-f_{i}, \qquad (i=1,2)
 \]
where $\mathcal{F}_{ii}$ denotes the Fermi-Dirac distribution for same-species interaction: 
\begin{align*}
	\mathcal{F}_{11}=\frac{1}{e^{m_1a_1\big|\frac{p}{m_1}-b_1\big|^2+c_1}+1},\quad
	\mathcal{F}_{22}=\frac{1}{e^{m_2a_2\big|\frac{p}{m_2}-b_2\big|^2+c_2}+1},
\end{align*}
 and $\mathcal{F}_{ij}$ denote Fermi-Dirac distribution for inter-species interactions:
\begin{align*}
	\mathcal{F}_{12}=\frac{1}{e^{m_1a\big|\frac{p}{m_1}-b\big|^2+c_{12}}+1},\quad
	\mathcal{F}_{21}=\frac{1}{e^{m_2a\big|\frac{p}{m_2}-b\big|^2+c_{21}}+1}.
\end{align*}
\item Boson-Boson interaction ($i\neq j$):
\[
\mathcal{R}_{ii}=\mathcal{B}_{ii}-f_{i},\quad
\mathcal{R}_{ij}=\mathcal{B}_{ij}-f_{i}, \qquad (i=1,2)
\]
where $\mathcal{B}_{ii}$ denotes the Bose-Einstein distribution for same-species interaction : 
\begin{align*}
	\mathcal{B}_{11}=\frac{1}{e^{m_1a_1\big|\frac{p}{m_1}-b_1\big|^2+c_1}-1},\quad
	\mathcal{B}_{22}=\frac{1}{e^{m_2a_2\big|\frac{p}{m_2}-b_2\big|^2+c_2}-1},
\end{align*}
while $\mathcal{B}_{ij}$ denote Bose-Einstein distribution for inter-species interactions:
\begin{align*}
\mathcal{B}_{12}=\frac{1}{e^{m_1a\big|\frac{p}{m_1}-b\big|^2+c_{12}}-1},\quad
\mathcal{B}_{21}=\frac{1}{e^{m_2a\big|\frac{p}{m_2}-b\big|^2+c_{21}}-1}.
\end{align*}
 
\item Fermion ($f_1$)-Boson ($f_2$) interaction:
\[
\mathcal{R}_{11}=\mathcal{F}_{11}-f_{1} \quad 
\mathcal{R}_{22}=\mathcal{B}_{22}-f_{2},
\]
and
\[
\mathcal{R}_{12}=\mathcal{F}_{12}-f_{1} \quad 
\mathcal{R}_{21}=\mathcal{B}_{21}-f_{2}.
\]

 \end{itemize}
For later convenience, and to unify the proof, we introduce the following notation for quantum equilibriums:
\begin{center}
$\bullet$ {\bf The quantum equilibrium $\mathcal{M}_{ij}$}
\end{center}
 Next, we will make statements on the equilibrium distributions in the relaxation operators that correspond to $\mathcal{F}_{ij}$ in the fermion case and $\mathcal{B}_{ij}$ in the boson case. In order not to list all different cases separately, we denote the equilibrium distribution by $\mathcal{M}_{ij}$ which is equal to a Fermi-Dirac or a Bose-Einstein distribution depending on the case we consider:
\begin{enumerate}
\item Fermion-Fermion interaction 
\begin{align*}
\mathcal{M}_{ij}=\mathcal{F}_{ij}.\quad(i,j=1,2)
\end{align*}
\item Boson-Boson interaction 
\begin{align*}
\mathcal{M}_{ij}=\mathcal{B}_{ij}.\quad(i,j=1,2)
\end{align*}
\item Fermion ($f_1$) - Boson ($f_2$) interaction 
\begin{align*}
	\mathcal{M}_{1j}=\mathcal{F}_{1j}, \quad \mathcal{M}_{2j}=\mathcal{B}_{2j}.\quad (j=1,2)
\end{align*}
\end{enumerate}

The excessive computational cost has already been a very serious obstacle even for the classical Boltzmann equation.
Since the difficulty mostly lies in the computation of the collision operator, various efforts to approximate the complicated
collision process with a numerically more amenable model have been made. The BGK model is introduced 
in \cite{bhathnagor1954model} as a result of such efforts, and now become the most popular
approximate model of the Boltzmann equation because it provides a very reliable results in a wide range of kinetic-fluid regime covering much of the practical problems at relatively low computational costs. 

As in the classical case, the quantum BGK models are widely used in place of the quantum Boltzmann equation. However, the quantum BGK model for mixture has not been rigorously studied yet. More precisely, 
whether the relaxation operator can be soundly defined in a rigorous manner so that it satisfies the same
conservation laws and the $H$-theorem as the quantum Boltzmann has never been rigorously verified in the literature. 
The well-definedness of such equilibrium coefficients for $\mathcal{M}_{11}$ and $\mathcal{M}_{22}$ follows directly from the relevant results for the one-species quantum BGK model in \cite{MR4064205,MR4096124,MR2145021,MR1751703,MR2378854}. Thus, we focus on the determination of the equilibrium coefficients for the mixture equilibrium $\mathcal{M}_{12}$ and $\mathcal{M}_{21}$.

\subsection{Determination of $\mathcal{M}_{ij}$ $(i,j=1,2)$}

The quantum BGK model may be far more amenable in terms of numerical computation, but the highly non-linear nature of the QBGK model gives rise to various difficulties in the analysis of the model.
As such, it turns out that the requirement that the QBGK model must share the conservation laws and $H$-theorem with the quantum Boltzmann equation, leads to a set of very complicated nonlinear relations for the equilibrium coefficients (See Section 2.2). Moreover, they involve different conditions of solvability according to the nature of the interactions: Fermion-Fermion interaction, Fermion-Boson interaction, Boson-Boson interaction.

In this paper, we explicitly derive the nonlinear relations among the equilibrium coefficients of $\mathcal{M}_{11}$, $\mathcal{M}_{22}$, $\mathcal{M}_{12}$, $\mathcal{M}_{21}$ that arise from the physical requirement of the equation, and  verify in a unified way that those nonlinear relations uniquely determined the coefficients under certain conditions.

First, we note that we need to determine the mixture local equilibrium $\mathcal{M}_{ij}$ in such way that the relaxation operator in the r.h.s of (\ref{MQBGK}) satisfies the same cancellation properties as (\ref{cancellation}) and the entropy dissipation in (\ref{entropy}) are determined by following conservation laws.

To be more specific, let $N_i$, $P_i$ and $E_i$ $(i=1,2)$ denote
\begin{align*}
	N_i=\int_{\mathbb{R}^3}f_idp,\quad P_i=\int_{\mathbb{R}^3}f_ipdp, \quad E_i=\int_{\mathbb{R}^3}f_i\frac{|p|^2}{2m_i}dp .
\end{align*}
Assuming that the r.h.s of (\ref{MQBGK}) satisfies the same identities in (\ref{cancellation}), we arrive at the following identities:\newline
\begin{align}\label{conserv 1}
	\int_{\mathbb{R}^3}\mathcal{M}_{ii}dp=N_i, \qquad \int_{\mathbb{R}^3}\mathcal{M}_{ii}pdp=P_i, \qquad \int_{\mathbb{R}^3}\mathcal{M}_{ii}\frac{|p|^2}{2m_i}dp=E_i, \quad (i=1,2)
\end{align}
and 
\begin{align}\label{conserv 2}
	\begin{split}
		&\int_{\mathbb{R}^3}\mathcal{M}_{12}dp=N_1, \qquad \int_{\mathbb{R}^3}\mathcal{M}_{21}dp=N_2, \cr
		&\int_{\mathbb{R}^3}\mathcal{M}_{12}pdp+\int_{\mathbb{R}^3}\mathcal{M}_{21}pdp=P_1+P_2, \cr
		&\int_{\mathbb{R}^3}\mathcal{M}_{12}\frac{|p|^2}{2m_1}dp+\int_{\mathbb{R}^3}\mathcal{M}_{21}\frac{|p|^2}{2m_2}dp=E_1+E_2.
	\end{split}
\end{align}

Our goal is to show that, for each fixed $N_i$, $P_i$, $E_i$ $(i=1,2)$, the relations in \eqref{conserv 1} and \eqref{conserv 2} completely and uniquely determine  $\mathcal{M}_{ij}$, which is stated in Theorem \ref{main result QQ}.

\subsection{Literature review: Quantum BGK models}

The quantum modification of the celebrated Boltzmann equation,
which is often called Uehling-Uhlenbeck equation or Nordheim equation in the literature,
was made in \cite{fowler1928electron,kikuchi1930kinetische,uehling1933transport,uehling1934transport} and soon recognized as a fundamental equation to describe quantum particles at mesoscopic level.
But due to the complexity of the collision operator, which is a serious obstacle to practical application of the equation,
and relaxation time approximations, or quantum BGK models are widely used to understand the transport phenomena and compute transport coefficients for semi-conductor device and crystal lattice \cite{ashcroft1976solid,duan2005introduction,jungel2009transport,MR2786396,MR1084373,MR1063852,MR3323935,ihn2004electronic,rapp2010equilibration} and various flow problems involving quantum effects \cite{MR3131732,MR2855649,filbet2010numerical,MR2786396,MR2467628,MR2923847,yang2009lattice,mcgaughey2004quantitative,schneider2012fermionic,sparavigna2016boltzmann}. For the development of numerical methods for quantum BGK model, we refer to  \cite{MR3131732,filbet2010numerical,MR2855649,MR2892305,MR1489261,suh1991numerical,MR2467628,MR2923847,yang2009lattice,MR3581504}.
We mention that the prototype of relaxation type models in quantum theory can be traced back to the Drude model \cite{drude1900elektronentheorie,drude1900elektronentheorie2} which successfully explained the fundamental transport property of electrons such as the Ohm's law or Hall effect. 

Mathematical study on the quantum BGK model is in its initial state.
Nouri studied the existence of weak solutions for a stationary quantum BGK model with a discretized condensation term in \cite{MR2378854}. Braukhoff \cite{MR3918275,braukhoff2018global}  established the existence of analytic solutions and studied its asymptotic behaviour for a quantum BGK type model describing the dynamics of the ultra-cold atoms in an optical lattice. Bae and Yun considered the existence and asymptotic stability of a fermionic quantum BGK model near a global Fermi-Dirac distribution in \cite{MR4096124}.\newline


\noindent{\bf BGK models for gas mixtures:}
There are many BGK models for gas mixtures proposed in the literature. 
Examples include the model of Gross and Krook \cite{gross1956model}, the model of Hamel \cite{hamel1965kinetic}, the model of Greene \cite{greene1973improved}, the model of Garzo, Santos and Brey \cite{garzo1989kinetic}, the model of Sofonea and Sekerka \cite{sofonea2001bgk}, the model by Andries, Aoki and Perthame  \cite{MR1889599}, the model of Brull, Pavan and Schneider \cite{MR2896732}, the model of Klingenberg, Pirner and Puppo \cite{MR3579578}, the model of Haack, Hauck, Murillo \cite{MR3680629} and the model of Bobylev, Bisi, Groppi, Spiga \cite{MR3815148}. 
BGK models have also been extended to ES-BGK models, polyatomic molecules or chemically reactive gas mixtures; see for example \cite{MR3436242,bisi2010kinetic,groppi2011kinetic,MR3960644,MR3828279,MR3880999, MR3943439,MR0258399,MR3720827}. 
BGK  models are often used in applications because they give rise to efficient numerical computations as compared to models with Boltzmann collision terms \cite{MR2403867,MR3411286,MR3231759,MR3202241,MR2320557, MR2674294,MR2970736,MR3892427}.

In the following Section 2.1, we state our main result. In Section 2.2, we derive a set of nonlinear 
functional relations and show that the equilibirum coefficients can be uniquely determined to satisfy 
the conservation laws of mass, momentum and energy. In Section 2.3, the BGK model defined with the equilibrium coefficients derived in Section 2.2, also satisfies the $H$-theorem.

\section{Determination of the relaxation operators for quantum mixture}
\subsection{Main result for general quantum-quantum interaction}
We now state our main result stating that the equilibrium coefficients, under appropriate
assumptions on $N_i$, $P_i$ and $E_i$, can be uniquely determined. To simplify the presentation,
we introduce $h_{\pm1}$, $j_{\pm1}$, $k$ by

\begin{align*} 
	h_{\pm1}(x)=\int_{\mathbb{R}^3}\frac{1}{e^{|p|^2+x}\pm1}dp,\qquad j_{\pm1}(x)=\frac{\int \frac{1}{e^{|p|^2+x}\pm1}dp}{\left(\int \frac{|p|^2}{e^{|p|^2+x}\pm1}dp\right)^{3/5}},  	
\end{align*}
and
\begin{align*}
k_{\tau,\tau'}(x,y)= \frac{m_1^{\frac{3}{2}}\int_{\mathbb{R}^3}\frac{1}{e^{|p|^2+x}+\tau}dp}{\left(m_1^{\frac{3}{2}}\int_{\mathbb{R}^3}\frac{|p|^2}{e^{|p|^2+x}+\tau}dp+m_2^{\frac{3}{2}}\int_{\mathbb{R}^3}\frac{|p|^2}{e^{|p|^2+y}+\tau'}dp\right)^\frac{3}{5}},
\end{align*}	
where the pair $(\tau,\tau')$ is chosen as follows:  
\begin{align*}
	(\tau,\tau')=\left\{
	\begin{array}{ll}
		(+1,+1)& \text{(fermion-fermion)}\cr
		(-1,-1)&\mbox{(boson-boson)}\cr
		(+1,-1)&\mbox{(fermion-boson)}\cr
		\end{array}\right.
\end{align*}

Using $h$ and $k$, we define $g$, which is defined as a composite function of $k$ and $h^{-1}$, as follows:
\begin{align}\label{QQg(x)}
g_{\tau,\tau'}(x)=k_{\tau,\tau'}\big(x, y(x)\big) =\frac{m_1^{\frac{3}{2}}\int_{\mathbb{R}^3}\frac{1}{e^{|p|^2+x}+\tau}dp}{\left(m_1^{\frac{3}{2}}\int_{\mathbb{R}^3}\frac{|p|^2}{e^{|p|^2+x}+\tau}dp+m_2^{\frac{3}{2}}\int_{\mathbb{R}^3}\frac{|p|^2}{e^{|p|^2+y(x)}+\tau'}dp\right)^\frac{3}{5}},
\end{align}
where $y(x)$ denotes
\begin{align*}
	y(x)=h_{\tau'}^{-1}\left(\frac{m_1^{\frac{3}{2}}N_2}{m_2^{\frac{3}{2}}N_1}h_{\tau}(x)\right).
\end{align*}
Note that $h_{\pm1}^{-1}$ always exist since $h_{\pm1}$ is strictly decreasing. For simplicity of notation, we define $l:\{+1,-1\}\rightarrow[-\infty,\infty]$ by
\begin{align*}
l(x)=\left\{
\begin{array}{l}
l(+1)=-\infty, \cr
l(-1)=0.\cr
\end{array}\right.
\end{align*}
In the following theorem, $j_{+1}(-\infty)$ is understood in the following sense:
 \[j_{+1}(-\infty)=\displaystyle{\lim_{x\rightarrow -\infty}j_{+1}(x)}.\]
 We note from \cite{MR4064205,MR4096124,MR1861208} that
\begin{align*}
\lim_{x\rightarrow -\infty}j_{+1}(x)=\frac{(4\pi)^{\frac{2}{5}}5^{\frac{3}{5}}}{3}. 
\end{align*}
%
%
%
%
\begin{theorem}\label{main result QQ}
	$(1)$ Assume, 
	\begin{align*}
		\frac{N_1}{\left(2m_1E_1-{P_1^2}/{N_1}\right)^{\frac{3}{5}}} \leq j_{\tau}(l(\tau)), \qquad 	\frac{N_2}{\left(2m_2E_2-{P_2^2}/{N_2}\right)^{\frac{3}{5}}} \leq j_{\tau'}(l(\tau')).
	\end{align*}
Then, we can define $c_i$ $(i=1,2)$ as the unique solution of
\begin{align*}
	j_{\tau}(c_1)=\frac{N_1}{\left(2m_1E_1-{|P_1|^2}/{N_1}\right)^{\frac{3}{5}}}, \qquad j_{\tau'}(c_2)=\frac{N_2}{\left(2m_2E_2-{|P_2|^2}/{N_1}\right)^{\frac{3}{5}}}.
\end{align*}
With $c_1$, $c_2$ obtained above, we then define $a_i$ $(i=1,2)$ by 
\begin{align*}
	a_1=m_1\left({\int_{\mathbb{R}^3}\frac{1}{e^{|p|^2+c_1}+\tau}dp}\right)^\frac{2}{3}N_1^{-\frac{2}{3}},\qquad a_2=m_2\left({\int_{\mathbb{R}^3}\frac{1}{e^{|p|^2+c_2}+\tau'}dp}\right)^\frac{2}{3}N_2^{-\frac{2}{3}},
\end{align*}
and
\begin{align*}
	b_1=\frac{P_1}{m_1N_1},\quad b_2=\frac{P_2}{m_2N_2}.
\end{align*}
Then, with such choice of $a_i$, $b_i$ and $c_i$,  $\mathcal{M}_{11}$ and $\mathcal{M}_{22}$ satisfies (\ref{conserv 1}).\newline

\noindent\noindent$(2)$ Assume further that
\begin{align*}
\frac{N_1}{\left(2E_1+2E_2-\frac{|P_1+P_2|^2}{m_1N_1+m_2N_2}\right)^{\frac{3}{5}}}\leq g_{\tau,\tau'}\left(\max\left\{l(\tau), h_{\tau}^{-1}\left(\frac{m_2^{\frac{3}{2}}N_1}{m_1^{\frac{3}{2}}N_2}h_{\tau'}(l(\tau'))\right)\right\} \right).
\end{align*}

Then $c_{12}$, $c_{21}$ are defined as a unique solution of the following relations:
\begin{align*}
	\frac{m_1^{\frac{3}{2}}h_{\tau}(c_{12})}{m_2^{\frac{3}{2}}h_{\tau'}(c_{21})}=\frac{N_1}{N_2}, \quad
	k_{\tau,\tau'}(c_{12}, c_{21})=\frac{N_1}{\left(2E_1+2E_2-\frac{|P_1+P_2|^2}{m_1N_1+m_2N_2}\right)^{\frac{3}{5}}}.
\end{align*}
With such $c_{12}$ and $c_{21}$, we define $a$ and $b$ by 
\begin{align*}
	a=\left(\frac{m_1^{\frac{3}{2}}\int_{\mathbb{R}^3}\frac{|p|^2}{e^{|p|^2+c_{12}}+\tau}dp+m_2^{\frac{3}{2}}\int_{\mathbb{R}^3}\frac{|p|^2}{e^{|p|^2+c_{21}}+\tau'}dp}{2E_1+2E_2-\frac{|P_1+P_2|^2}{m_1N_1+m_2N_2}}\right)^{\frac{2}{5}},\quad
	b=\frac{P_1+P_2}{m_1N_1+m_2N_2},
\end{align*}
Then, with these choices of equilibrium coefficients, our quantum BGK model for gas mixture (\ref{MQBGK}) satisfies (\ref{conserv 2}).\newline

\noindent$(3)$ With the choice of equilibrium coefficients as in (1), (2), the quantum BGK model for gas mixture (\ref{MQBGK}) satisfies the $H$-theorem. The equality in the $H$-Theorem is characterized by $f_1$ and $f_2$ being two Fermion distributions in the Fermion-Fermion case, two Bose distributions in the Boson- Boson case and a Fermion distribution and a Bose distribution in the Fermion-Boson case. In all the cases, these equilibrium distributions have the same $a$ and $b$.
\end{theorem}

%
%
%
\subsection{Proof of Theorem \ref{main result QQ} (1), (2)}
The proof for (1) can be found in \cite{MR4064205}. Therefore, we start with the proof of (2).
An explicit computation from $(\ref{conserv 2})_2$ gives 
\begin{align*}
	P_1(x,t)+P_2(x,t)&=\int_{\mathbb{R}^3}\frac{p}{e^{m_1a\big|\frac{p}{m_1}-b\big|^2+c_{12}}+\tau}dp+\int_{\mathbb{R}^3}\frac{p}{e^{m_2a\big|\frac{p}{m_2}-b\big|^2+c_{21}}+\tau'}dp	\cr
	&=\int_{\mathbb{R}^3}\frac{p+m_1b}{e^{a|p|^2+c_{12}}+\tau}dp+\int_{\mathbb{R}^3}\frac{p+m_2b}{e^{a|p|^2+c_{21}}+\tau'}dp \cr
	&= b (m_1N_1(x,t)+m_2N_2(x,t)).
\end{align*}
This gives the explicit presentation of  $b$: 
\begin{align}\label{QQb}
	b(x,t)=\frac{P_1(x,t)+P_2(x,t)}{m_1N_1(x,t)+m_2N_2(x,t)}.
\end{align}
On the other hand, we have from $(\ref{conserv 2})_1$ that:
\begin{align}\label{QQN1}
	\begin{split}
		N_1(x,t)=\int_{\mathbb{R}^3}\frac{1}{e^{m_1a\big|\frac{p}{m_1}-b\big|^2+c_{12}}+\tau}dp	
		&=m_1^{\frac{3}{2}}a^{-\frac{3}{2}}\int_{\mathbb{R}^3}\frac{1}{e^{|p|^2+c_{12}}+\tau}dp, \cr
		N_2(x,t)=\int_{\mathbb{R}^3}\frac{1}{e^{m_2a\big|\frac{p}{m_2}-b\big|^2+c_{21}}+\tau'}dp	
		&=m_2^{\frac{3}{2}}a^{-\frac{3}{2}}\int_{\mathbb{R}^3}\frac{1}{e^{|p|^2+c_{21}}+\tau'}dp,
	\end{split}
\end{align}
and from $(\ref{conserv 2})_3$:
\begin{align}\label{QQE+E}
	\begin{split}
	E_1(x,t)+E_2(x,t)&=\frac{1}{2m_1}\int_{\mathbb{R}^3}\frac{|p|^2}{e^{m_1a\big|\frac{p}{m_1}-b\big|^2+c_{12}}+\tau}dp+\frac{1}{2m_2}\int_{\mathbb{R}^3}\frac{|p|^2}{e^{m_2a\big|\frac{p}{m_2}-b\big|^2+c_{21}}+\tau'}dp	\cr
	&= \frac{1}{2}m_1^{\frac{3}{2}}a^{-\frac{5}{2}}\int_{\mathbb{R}^3}\frac{|p|^2}{e^{|p|^2+c_{12}}+\tau}dp+\frac{1}{2}m_2^{\frac{3}{2}}a^{-\frac{5}{2}}\int_{\mathbb{R}^3}\frac{|p|^2}{e^{|p|^2+c_{21}}+\tau'}dp\cr&+\frac{1}{2}(m_1N_1+m_2N_2)b^2(x,t),
\end{split}
\end{align}
Plugging (\ref{QQb}) into (\ref{QQE+E}), we get
\begin{align}\label{QQE12}
2E_1+2E_2-\frac{|P_1+P_2|^2}{m_1N_1+m_2N_2}&=a^{-\frac{5}{2}}\left(m_1^{\frac{3}{2}}\int_{\mathbb{R}^3}\frac{|p|^2}{e^{|p|^2+c_{12}}+\tau}dp+m_2^{\frac{3}{2}}\int_{\mathbb{R}^3}\frac{|p|^2}{e^{|p|^2+c_{21}}+\tau'}dp\right)
\end{align}
We then deduce from \eqref{QQE12} and $\eqref{QQN1}_1$ that
\begin{align}\label{QQc1}
	\frac{N_1}{\left(2E_1+2E_2-\frac{|P_1+P_2|^2}{m_1N_1+m_2N_2}\right)^{\frac{3}{5}}}= \frac{m_1^{\frac{3}{2}}\int_{\mathbb{R}^3}\frac{1}{e^{|p|^2+c_{12}}+\tau}dp}{{\left(m_1^{\frac{3}{2}}\int_{\mathbb{R}^3}\frac{|p|^2}{e^{|p|^2+c_{12}}+\tau}dp+m_2^{\frac{3}{2}}\int_{\mathbb{R}^3}\frac{|p|^2}{e^{|p|^2+c_{21}}+\tau'}dp\right)^\frac{3}{5}}},
\end{align}
On the other hand, we can factor out $a$ by dividing the two relations in \eqref{QQN1}:
\begin{align}\label{QQN12}
\frac{N_1}{N_2}=\frac{m_1^{\frac{3}{2}}\int_{\mathbb{R}^3}\frac{1}{e^{|p|^2+c_{12}}+\tau}dp}{m_2^{\frac{3}{2}}\int_{\mathbb{R}^3}\frac{1}{e^{|p|^2+c_{21}}+\tau'}dp} = \frac{m_1^{\frac{3}{2}}h_{\tau}(c_{12})}{m_2^{\frac{3}{2}}h_{\tau'}(c_{21})}
\end{align}
and hence:
\begin{align}\label{QQc21}
	c_{21}=h_{\tau'}^{-1}\left(\frac{m_1^{\frac{3}{2}}N_2}{m_2^{\frac{3}{2}}N_1}h_{\tau}(c_{12})\right),
\end{align}
from the monotonicity of $h_{\tau}$.
Now, considering that  $a$ is  obtained from
(\ref{QQE12}) once $c_{12}$ and $c_{21}$ are chosen,
it remains, under the assumption of Theorem \ref{main result QQ}, that \eqref{QQc1} and \eqref{QQN12} uniquely determine $c_{12}$ and $c_{21}$. 
In turn, in view of (\ref{QQc1}) and $(\ref{QQc21})$, we see that $c_{12}$ and $c_{21}$ can be uniquely determined once we prove the monotonicity of $g_{\tau,\tau'}$, which is stated in the following lemma.

\begin{lemma}\label{gmono} Recall the definition of $g_{\tau,\tau'}$ given in (\ref{QQg(x)}):
\begin{align*}
	g_{\tau,\tau'}(x)= \frac{m_1^{\frac{3}{2}}\int_{\mathbb{R}^3}\frac{1}{e^{|p|^2+x}+\tau}dp}{\left(m_1^{\frac{3}{2}}\int_{\mathbb{R}^3}\frac{|p|^2}{e^{|p|^2+x}+\tau}dp+m_2^{\frac{3}{2}}\int_{\mathbb{R}^3}\frac{|p|^2}{e^{|p|^2+y(x)}+\tau'}dp\right)^\frac{3}{5}},
\end{align*}
where
\begin{align}\label{y}
	y(x)=h_{\tau'}^{-1}\left(\frac{m_1^{\frac{3}{2}}N_2}{m_2^{\frac{3}{2}}N_1}h_{\tau}(x)\right),
\end{align}
Then $g_{\tau,\tau'}(x)$ is strictly monotone decreasing function when $x\geq \max\left\{l(\tau), h_{\tau}^{-1}\left(\frac{m_2^{\frac{3}{2}}N_1}{m_1^{\frac{3}{2}}N_2}h_{\tau'}(l(\tau'))\right)\right\} $.
\end{lemma}
\begin{proof}
{\bf Claim :} We claim that the following identity holds: 
\begin{align}\label{identity}
\displaystyle g'_{\tau,\tau'}(x)=8\pi^2  \frac{\displaystyle m_1^3D_{\tau}(x)+m_1^{\frac{3}{2}}m_2^{\frac{3}{2}}\frac{\int_0^{\infty}\frac{1}{e^{r^2+x}+\tau}dr}{\int_0^{\infty}\frac{1}{e^{r^2+y(x)}+\tau'}dr}D_{\tau'}(y(x))}
{\displaystyle\left(m_1^{\frac{3}{2}}\int_{\mathbb{R}^3}\frac{|p|^2}{e^{|p|^2+x}+\tau}dp+m_2^{\frac{3}{2}}\int_{\mathbb{R}^3}\frac{|p|^2}{e^{|p|^2+y(x)}+\tau'}dp\right)^{\frac{8}{5}}}
\end{align}
where
\begin{align*}
	D_{\tau}(x) =  \frac{9}{5}\int_0^{\infty}\frac{r^2}{e^{r^2+x}+\tau}dr\int_{0}^{\infty}\frac{r^2}{e^{r^2+x}+\tau}dr - \int_0^{\infty}\frac{r^4}{e^{r^2+x}+\tau}dr\int_0^{\infty}\frac{1}{e^{r^2+x}+\tau}dr.
\end{align*}

\noindent$\bullet$ {\bf Proof of (\ref{identity}):}  By an explicit computation, we have 
\begin{align*}
	\frac{\partial g(x)}{\partial x}  &= {\left(m_1^{\frac{3}{2}}\int_{\mathbb{R}^3}\frac{|p|^2}{e^{|p|^2+x}+\tau}dp+m_2^{\frac{3}{2}}\int_{\mathbb{R}^3}\frac{|p|^2}{e^{|p|^2+y(x)}+\tau'}dp\right)^{-\frac{6}{5}}} \cr
	&\times \bigg[ {\left(m_1^{\frac{3}{2}}\int_{\mathbb{R}^3}\frac{|p|^2}{e^{|p|^2+x}+\tau}dp+m_2^{\frac{3}{2}}\int_{\mathbb{R}^3}\frac{|p|^2}{e^{|p|^2+y(x)}+\tau'}dp\right)^\frac{3}{5}}m_1^{\frac{3}{2}}\partial_{x}\int_{\mathbb{R}^3}\frac{1}{e^{|p|^2+x}+\tau}dp\cr
	&-\frac{3}{5}\left(m_1^{\frac{3}{2}}\int_{\mathbb{R}^3}\frac{|p|^2}{e^{|p|^2+x}+\tau}dp+m_2^{\frac{3}{2}}\int_{\mathbb{R}^3}\frac{|p|^2}{e^{|p|^2+y(x)}+\tau'}dp\right)^{-\frac{2}{5}}\cr
	&\times \partial_x\left(m_1^{\frac{3}{2}}\int_{\mathbb{R}^3}\frac{|p|^2}{e^{|p|^2+x}+\tau}dp+m_2^{\frac{3}{2}}\int_{\mathbb{R}^3}\frac{|p|^2}{e^{|p|^2+y(x)}+\tau'}dp\right)  m_1^{\frac{3}{2}}\int_{\mathbb{R}^3}\frac{1}{e^{|p|^2+x}+\tau}dp \bigg].
\end{align*}
We then multiply $2/5$ power of \[
m_1^{\frac{3}{2}}\int_{\mathbb{R}^3}\frac{|p|^2}{e^{|p|^2+x}+\tau}dp+m_2^{\frac{3}{2}}\int_{\mathbb{R}^3}\frac{|p|^2}{e^{|p|^2+y(x)}+\tau'}dp\] 
on numerator and denominator:
\begin{align*}
	\frac{\partial g(x)}{\partial x}  &= {\left(m_1^{\frac{3}{2}}\int_{\mathbb{R}^3}\frac{|p|^2}{e^{|p|^2+x}+\tau}dp+m_2^{\frac{3}{2}}\int_{\mathbb{R}^3}\frac{|p|^2}{e^{|p|^2+y(x)}+\tau'}dp\right)^{-\frac{8}{5}}} \cr
	&\times \bigg[ \left(m_1^{\frac{3}{2}}\int_{\mathbb{R}^3}\frac{|p|^2}{e^{|p|^2+x}+\tau}dp+m_2^{\frac{3}{2}}\int_{\mathbb{R}^3}\frac{|p|^2}{e^{|p|^2+y(x)}+\tau'}dp\right)m_1^{\frac{3}{2}}\partial_{x}\int_{\mathbb{R}^3}\frac{1}{e^{|p|^2+x}+\tau}dp\cr
	&-\frac{3}{5}\partial_x\left(m_1^{\frac{3}{2}}\int_{\mathbb{R}^3}\frac{|p|^2}{e^{|p|^2+x}+\tau}dp+m_2^{\frac{3}{2}}\int_{\mathbb{R}^3}\frac{|p|^2}{e^{|p|^2+y(x)}+\tau'}dp\right)m_1^{\frac{3}{2}}\int_{\mathbb{R}^3}\frac{1}{e^{|p|^2+x}+\tau}dp \bigg].
\end{align*}
We then set the denominator to be $I$ to write
\begin{align*}
	\frac{\partial g(x)}{\partial x}  &= {\left(m_1^{\frac{3}{2}}\int_{\mathbb{R}^3}\frac{|p|^2}{e^{|p|^2+x}+\tau}dp+m_2^{\frac{3}{2}}\int_{\mathbb{R}^3}\frac{|p|^2}{e^{|p|^2+y(x)}+\tau'}dp\right)^{-\frac{8}{5}}} \times I,
\end{align*}
where
\begin{align*}
	I&= \left(m_1^{\frac{3}{2}}\int_{\mathbb{R}^3}\frac{|p|^2}{e^{|p|^2+x}+\tau}dp+m_2^{\frac{3}{2}}\int_{\mathbb{R}^3}\frac{|p|^2}{e^{|p|^2+y(x)}+\tau'}dp\right)m_1^{\frac{3}{2}}\partial_{x}\int_{\mathbb{R}^3}\frac{1}{e^{|p|^2+x}+\tau}dp\cr
	&-\frac{3}{5}\partial_x\left(m_1^{\frac{3}{2}}\int_{\mathbb{R}^3}\frac{|p|^2}{e^{|p|^2+x}+\tau}dp+m_2^{\frac{3}{2}}\int_{\mathbb{R}^3}\frac{|p|^2}{e^{|p|^2+y(x)}+\tau'}dp\right)m_1^{\frac{3}{2}}\int_{\mathbb{R}^3}\frac{1}{e^{|p|^2+x}+\tau}dp.
\end{align*}
We then carry out the following two integrations
\begin{align}\label{int1}
	\begin{split}
	\partial_{x}\int_{\mathbb{R}^3}\frac{1}{e^{|p|^2+x}+\tau}dp&= \int_{\mathbb{R}^3}\frac{-e^{|p|^2+x}}{(e^{|p|^2+x}+\tau)^2}dp\cr
	& = 4\pi \int_{0}^{\infty}\frac{-r^2e^{r^2+x}}{(e^{r^2+x}+\tau)^2}dr \cr
	&= -2\pi\int_0^{\infty}\frac{1}{e^{r^2+x}+\tau}dr
	\end{split}
\end{align}
where we used the following integration by parts : $u'=\frac{2re^{r^2+x}}{(e^{r^2+x}+\tau)^2}$, $v=\frac{1}{2}r$,
and
\begin{align}\label{int2}
	\begin{split}
	\partial_x&\left(m_1^{\frac{3}{2}}\int_{\mathbb{R}^3}\frac{|p|^2}{e^{|p|^2+x}+\tau}dp+m_2^{\frac{3}{2}}\int_{\mathbb{R}^3}\frac{|p|^2}{e^{|p|^2+y(x)}+\tau'}dp\right)\cr
	&=m_1^{\frac{3}{2}}\int_{\mathbb{R}^3}\frac{-|p|^2e^{|p|^2+x}}{(e^{|p|^2+x}+\tau)^2}dp+m_2^{\frac{3}{2}}\frac{\partial y(x)}{\partial x}\int_{\mathbb{R}^3}\frac{-|p|^2e^{|p|^2+y(x)}}{(e^{|p|^2+y(x)}+\tau')^2}dp \cr
	&= 4\pi m_1^{\frac{3}{2}} \int_{0}^{\infty}\frac{-r^4e^{r^2+x}}{(e^{r^2+x}+\tau)^2}dr+4\pi m_2^{\frac{3}{2}}\frac{\partial y(x)}{\partial x} \int_{0}^{\infty}\frac{-r^4e^{r^2+y(x)}}{(e^{r^2+y(x)}+\tau')^2}dr \cr
	&= -6\pi m_1^{\frac{3}{2}} \int_0^{\infty}\frac{r^2}{e^{r^2+x}+\tau}dr-6\pi m_2^{\frac{3}{2}}\frac{\partial y(x)}{\partial x}\int_0^{\infty}\frac{r^2}{e^{r^2+y(x)}+\tau'}dr,
\end{split}
\end{align}
where we used similar integration by parts : $u'=\frac{2re^{r^2+c}}{(e^{r^2+c}+\tau)^2}$, $v=\frac{1}{2}r^3$ for
\begin{align*}
	\int_0^{\infty}\frac{r^4e^{r^2+c}}{(e^{r^2+c}+\tau)^2}dr=\frac{3}{2}\int_0^{\infty}\frac{r^2}{e^{r^2+c}+\tau}dr.
\end{align*}
Using (\ref{int1}) and (\ref{int2}), we rewrite $I$ as
\begin{align}\label{re}
	\begin{split}
	I&= -8\pi^2\left(m_1^{\frac{3}{2}}\int_0^{\infty}\frac{r^4}{e^{r^2+x}+\tau}dr+m_2^{\frac{3}{2}}\int_0^{\infty}\frac{r^4}{e^{r^2+y(x)}+\tau'}dr\right)m_1^{\frac{3}{2}}\int_0^{\infty}\frac{1}{e^{r^2+x}+\tau}dr\cr
	&+\frac{72\pi^2}{5}\left(m_1^{\frac{3}{2}}\int_0^{\infty}\frac{r^2}{e^{r^2+x}+\tau}dr+m_2^{\frac{3}{2}}\frac{\partial y(x)}{\partial x}\int_0^{\infty}\frac{r^2}{e^{r^2+y(x)}+\tau'}dr\right)m_1^{\frac{3}{2}}\int_{0}^{\infty}\frac{r^2}{e^{r^2+x}+\tau}dr 
\end{split}
\end{align}
We then recall 
\begin{align*}
	D_{\tau}(x)=-\int_0^{\infty}\frac{r^4}{e^{r^2+x}+\tau}dr\int_0^{\infty}\frac{1}{e^{r^2+x}+\tau}dr+\frac{9}{5}\int_0^{\infty}\frac{r^2}{e^{r^2+x}+\tau}dr\int_{0}^{\infty}\frac{r^2}{e^{r^2+x}+\tau}dr <0,
\end{align*}
and express (\ref{re}) as follows:
So subtracting $D_{\tau}(x)$ on each sides gives 
\begin{align}\label{turn back}
	\begin{split}
	\frac{I}{8\pi^2}-m_1^3D_{\tau}(x)&= -m_1^{\frac{3}{2}}m_2^{\frac{3}{2}}\int_0^{\infty}\frac{r^4}{e^{r^2+y(x)}+\tau'}dr\int_0^{\infty}\frac{1}{e^{r^2+x}+\tau}dr\cr
	&+\frac{9}{5}m_1^{\frac{3}{2}}m_2^{\frac{3}{2}}\frac{\partial y(x)}{\partial x}\int_0^{\infty}\frac{r^2}{e^{r^2+y(x)}+\tau'}dr\int_{0}^{\infty}\frac{r^2}{e^{r^2+x}+\tau}dr .
\end{split}
\end{align}
Now we compute $\partial y(x)/\partial x$. Recall 
\begin{align*}
y(x)=h_{\tau'}^{-1}\left(\frac{m_1^{\frac{3}{2}}N_2}{m_2^{\frac{3}{2}}N_1}h_{\tau}(x)\right),
\end{align*}
and compute
\begin{align*}
	\frac{d y(x)}{d x} &= (h_{\tau'}^{-1})'\left(\frac{m_1^{\frac{3}{2}}N_2}{m_2^{\frac{3}{2}}N_1}h_{\tau}(x)\right) \times \frac{d }{d x}\frac{m_1^{\frac{3}{2}}N_2}{m_2^{\frac{3}{2}}N_1}h_{\tau}(x).
\end{align*}
Then, since the differentiation rule for inverse function gives
\begin{align*}
	(h_{\tau'}^{-1})'\left(\frac{m_1^{\frac{3}{2}}N_2}{m_2^{\frac{3}{2}}N_1}h_{\tau}(x)\right) = \frac{1}{h_{\tau'}'(y(x))},
\end{align*}
we get
\begin{align*}
	\frac{d y(x)}{d x} &= \frac{m_1^{\frac{3}{2}}N_2}{m_2^{\frac{3}{2}}N_1}\frac{h_{\tau}'(x)}{h_{\tau'}'(y(x))}.
\end{align*}
Finally, we use
\begin{align*}
	h_{\tau}'(x) = \int_{\mathbb{R}^3}\frac{-e^{|p|^2+x}}{(e^{|p|^2+x}+\tau)^2}dp =4\pi \int_{0}^{\infty}\frac{-r^2e^{r^2+x}}{(e^{r^2+x}+\tau)^2}dr =-2\pi \int_0^{\infty}\frac{1}{e^{r^2+x}+\tau}dr,
\end{align*}
to obtain the following expressions for $\partial y/\partial x$:
\begin{align*}
	\frac{\partial y(x)}{\partial x} &= \frac{m_1^{\frac{3}{2}}N_2}{m_2^{\frac{3}{2}}N_1}\frac{\int_0^{\infty}\frac{1}{e^{r^2+x}+\tau}dr}{\int_0^{\infty}\frac{1}{e^{r^2+y(x)}+\tau'}dr}. 
\end{align*}
Inserting this into (\ref{turn back})
\begin{align*}
\frac{I}{8\pi^2}-m_1^3D_{\tau}(x)&= -m_1^{\frac{3}{2}}m_2^{\frac{3}{2}}\int_0^{\infty}\frac{r^4}{e^{r^2+y(x)}+\tau'}dr\int_0^{\infty}\frac{1}{e^{r^2+x}+\tau}dr\cr
&+\frac{9}{5}m_1^3\frac{N_2}{N_1}\frac{\int_0^{\infty}\frac{1}{e^{r^2+x}+\tau}dr}{\int_0^{\infty}\frac{1}{e^{r^2+y(x)}+\tau'}dr} \int_0^{\infty}\frac{r^2}{e^{r^2+y(x)}+\tau'}dr\int_{0}^{\infty}\frac{r^2}{e^{r^2+x}+\tau}dr \cr
&=-m_1^{\frac{3}{2}}\int_0^{\infty}\frac{1}{e^{r^2+x}+\tau}dr\bigg(m_2^{\frac{3}{2}}\int_0^{\infty}\frac{r^4}{e^{r^2+y(x)}+\tau'}dr\cr
&\quad -\frac{9}{5}m_1^{\frac{3}{2}}\frac{N_2}{N_1}\frac{\int_0^{\infty}\frac{r^2}{e^{r^2+y(x)}+\tau'}dr\int_{0}^{\infty}\frac{r^2}{e^{r^2+x}+\tau}dr}{\int_0^{\infty}\frac{1}{e^{r^2+y(x)}+\tau'}dr}  \bigg)
\end{align*}
Finally, we use
\begin{align*}
	\frac{N_2}{N_1}=\frac{m_2^{\frac{3}{2}}h_{\tau'}(y(x))}{m_1^{\frac{3}{2}}h_{\tau}(x)}=\frac{m_2^{\frac{3}{2}}\int_{\mathbb{R}^3}\frac{1}{e^{|p|^2+y(x)}+\tau'}dp}{m_1^{\frac{3}{2}}\int_{\mathbb{R}^3}\frac{1}{e^{|p|^2+x}+\tau}dp} =\frac{m_2^{\frac{3}{2}}\int_0^{\infty}\frac{r^2}{e^{r^2+y(x)}+\tau'}dr}{m_1^{\frac{3}{2}}\int_0^{\infty}\frac{r^2}{e^{r^2+x}+\tau}dr}
\end{align*}
to derive
\begin{align*}
&\frac{I}{8\pi^2}-m_1^3D_{\tau}(x)\cr
&\hspace{0.6cm}=-m_1^{\frac{3}{2}}m_2^{\frac{3}{2}}\int_0^{\infty}\frac{1}{e^{r^2+x}+\tau}dr \left(\int_0^{\infty}\frac{r^4}{e^{r^2+y(x)}+\tau'}dr-\frac{9}{5}\frac{\int_0^{\infty}\frac{r^2}{e^{r^2+y(x)}+\tau'}dr\int_{0}^{\infty}\frac{r^2}{e^{r^2+y(x)}+\tau'}dr}{\int_0^{\infty}\frac{1}{e^{r^2+y(x)}+\tau'}dr}  \right) \cr
&\hspace{0.6cm}=m_1^{\frac{3}{2}}m_2^{\frac{3}{2}}\frac{\int_0^{\infty}\frac{1}{e^{r^2+x}+\tau}dr}{\int_0^{\infty}\frac{1}{e^{r^2+y(x)}+\tau'}dr} \cr
&\hspace{0.6cm}\times  \left(\frac{9}{5}\int_0^{\infty}\frac{r^2}{e^{r^2+y(x)}+\tau'}dr\int_{0}^{\infty}\frac{r^2}{e^{r^2+y(x)}+\tau'}dr - \int_0^{\infty}\frac{r^4}{e^{r^2+y(x)}+\tau'}dr\int_0^{\infty}\frac{1}{e^{r^2+y(x)}+\tau'}dr\right)\cr
&\hspace{0.6cm}=m_1^{\frac{3}{2}}m_2^{\frac{3}{2}}\frac{\int_0^{\infty}\frac{1}{e^{r^2+x}+\tau}dr}{\int_0^{\infty}\frac{1}{e^{r^2+y(x)}+\tau'}dr} 
D_{\tau'}(y(x)),
\end{align*}
which complete the proof of the claim.\newline

\noindent$\bullet$ {\bf Proof of the Lemma \ref{gmono}:} Assume (\ref{identity}) holds. We first observe that $h(x)$ is strictly decreasing function on $x\in[0,\infty)$ for $\tau=-1$ and $x\in(-\infty,\infty)$ for $\tau=+1$ :
\begin{align*}
h_{\tau}'(x)=-\int_{\mathbb{R}^3}\frac{e^{|p|^2+x}}{(e^{|p|^2+x}+\tau)^2}dp < 0.
\end{align*}
Therefore, our restriction on $x$: $x\geq  h_{\tau}^{-1}\left(\frac{m_2^{\frac{3}{2}}N_1}{m_1^{\frac{3}{2}}N_2}h_{\tau'}(l(\tau'))\right)$ combined with the definition of $y$ given in  \eqref{y}, leads to 
\begin{align*}
y(x)\equiv h_{\tau'}^{-1}\left(\frac{m_1^{\frac{3}{2}}N_2}{m_2^{\frac{3}{2}}N_1}h_{\tau}(x)\right)
\geq h_{\tau'}^{-1}\left(\frac{m_1^{\frac{3}{2}}N_2}{m_2^{\frac{3}{2}}N_1}h_{\tau}\left(h_{\tau}^{-1}\left(\frac{m_2^{\frac{3}{2}}N_1}{m_1^{\frac{3}{2}}N_2}h_{\tau'}(l(\tau'))\right)\right)\right)
= l(\tau').
\end{align*}
Thus, we have
\[
y(x)\geq l(\tau') .
\]
On the other hand, from the assumption, $x$ satisfies 
\[
x\geq l(\tau).
\]
Therefore, we have
\begin{align*}
D_{\tau}(x)<0	\quad\mbox{and}\quad D_{\tau'}(y(x))<0,
\end{align*}
since we already know 
\begin{align*}
D_{+1}(x)&<0 \mbox{ on } x\in (-\infty,\infty), \qquad
D_{-1}(x)<0 \mbox{ on } x\in [0,\infty).
\end{align*}
(See \cite{MR1751703} for boson case ($+1$) and \cite{MR4064205,MR1861208} for fermion case ($-1$)).  Inserting this into (\ref{identity}), we can conclude the proof of the Lemma.\newline 
\end{proof}

\subsection{Proof of Theorem \ref{main result QQ} (3)}It remains to prove the $H$-theorem to conclude Theorem \ref{main result QQ} (3). 
\begin{proposition} Let $f_i\leq 1$ only when $f_i$ is the distribution function for fermion components, then we have
\begin{align*}
\int_{\mathbb{R}^3}\ln\frac{f_1}{1-\tau f_1} \big\{ (\mathcal{M}_{11}-f_1)+(\mathcal{M}_{12}-f_1)\big\}+\ln\frac{f_2}{1-\tau' f_2}\big\{(\mathcal{M}_{22}-f_2)+(\mathcal{M}_{21}-f_2)\big\} dp \leq 0.
\end{align*}
\end{proposition}
\begin{proof}
The proof for
\begin{align}
\int_{\mathbb{R}^3}\ln\frac{f_1}{1-f_1}(\mathcal{M}_{11}-f_1)dp + \int_{\mathbb{R}^3}\ln\frac{f_2}{1-f_2}(\mathcal{M}_{22}-f_2)dp \leq 0,
\end{align}
can be found in \cite{MR2923847}. So we only prove
\begin{align*}
	S\equiv\int_{\mathbb{R}^3}\ln\frac{f_1}{1-\tau f_1}(\mathcal{M}_{12}-f_1)dp + \int_{\mathbb{R}^3}\ln\frac{f_2}{1-\tau' f_2}(\mathcal{M}_{21}-f_2)dp \leq0.
\end{align*}
First,  we observe that 
\begin{align*}
	I=\int_{\mathbb{R}^3}\ln\frac{\mathcal{M}_{12}}{1-\tau \mathcal{M}_{12}}(\mathcal{M}_{12}-f_1)dp + \int_{\mathbb{R}^3}\ln\frac{\mathcal{M}_{21}}{1-\tau' \mathcal{M}_{21}}(\mathcal{M}_{21}-f_2)dp =0,
\end{align*}
which follows from an explicit computation using the conservation laws \eqref{conserv 2}: 
\begin{align*}	I&=-\int_{\mathbb{R}^3}\left(am_1\bigg|\frac{p}{m_1}-b\bigg|^2+c_{12}\right)(\mathcal{M}_{12}-f_1)dp - \int_{\mathbb{R}^3}\left(am_2\bigg|\frac{p}{m_2}-b\bigg|^2+c_{21}\right)(\mathcal{M}_{21}-f_2)dp \cr
&=a \int_{\mathbb{R}^3} \left(\frac{|p|^2}{m_1}f_1+\frac{|p|^2}{m_2}f_2-\frac{|p|^2}{m_1}\mathcal{M}_{12}-\frac{|p|^2}{m_2}\mathcal{M}_{21}\right)dp - 2ab\cdot \int_{\mathbb{R}^3}p \left(f_1+f_2-\mathcal{M}_{12}-\mathcal{M}_{21}\right)dp \cr
	&=0.
\end{align*}
From this, we find
\begin{align*}
	S-I &= \int_{\mathbb{R}^3}\left( \ln\frac{f_1}{1-\tau f_1}-\ln\frac{\mathcal{M}_{12}}{1-\tau \mathcal{M}_{12}}\right) (\mathcal{M}_{12}-f_1)dp \cr
	&+ \int_{\mathbb{R}^3}\left(\ln\frac{f_2}{1-\tau'f_2}-\ln\frac{\mathcal{M}_{21}}{1-\tau'\mathcal{M}_{21}}\right)(\mathcal{M}_{21}-f_2)dp \leq 0,
\end{align*}
since $\ln \frac{x}{1+x}$ is an increasing function for $x\in[0,\infty)$, and $\ln \frac{x}{1-x}$ is an increasing function when $0<x<1$. Here, we have equality if and only if $f_1=M_{12}$ and $f_2=M_{21}$.
This completes the proof.
\end{proof}
\begin{remark}
The equality in the $H$-Theorem is characterized by two distributions with the same value for $a$ and $b$.  Due to the fact that $b$ is equal to pressure over the density, this leads to $P_1=\frac{N_1}{N_2} P_2$.
\end{remark}
Therefore, to complete the proof of Theorem \ref{main result QQ} (3), it remains to prove that $f_i<1$ in the case of fermions. 
\begin{lemma} Let $f_i$ be a distribution function for fermions  and $f_{i}(x,p,0) < 1$. Then we have $f_i(x,p,t)<1$ for $t\geq0$.
\end{lemma}
\begin{proof} 
Integrating \eqref{MQBGK} along the characteristic, we get the mild form :  
\begin{align*}
	f_i(x,p,t)&=e^{-2t}f_i(x-pt,p,0)+\int_0^t e^{2(\tau-t)}(\mathcal{F}_{ii}+\mathcal{F}_{ij})(x+(\tau-t) p,p,\tau)d\tau ,
\end{align*}
for $j\neq i $. Since $\mathcal{F}_{ii}<1$ and $\mathcal{F}_{ij}<1$ for all $(x,p,t)$ by definition, we have 
\begin{align*}
	f_i(x,p,t)&\leq e^{-2t}f_i(x-pt,p,0)+\int_0^t 2e^{2(\tau-t)}d\tau \cr
	&= e^{-2t}f_i(x-pt,p,0)+(1-e^{-2t})\cr
	&<1.\cr
\end{align*}
\end{proof}

\section{Appendix}
In this section, we present a proof for \eqref{cancellation} for readers' convenience. The proof is standard but we couldn't locate them in the literature. We also present the relation between the consevation laws w.r.t the momentum distribution function $f(x,p,t)$ and the conservation laws w.r.t the velocitiy distribution function $\bar{f}(x,v,t)$. We start with the computation of Jacobian:
\begin{lemma}\label{J} The Jacobian of the change of variables $(p,p_*) \leftrightarrow (p',p_*')$ is 
	\begin{align*}
		\det J = \det \frac{\partial(p',p_*')}{\partial(p,p_*)} =  -1.
	\end{align*}	
\end{lemma}
\begin{proof}
	A direct computation gives 
	\begin{align*}
		J= \frac{\partial(p',p_*')}{\partial(p,p_*)}= \left[\begin{array}{cc} \delta_{ij}-\frac{2m_1m_2}{m_1+m_2}\frac{w_iw_j}{m_1} & \frac{2m_1m_2}{m_1+m_2}\frac{w_iw_j}{m_2} \cr
			\frac{2m_1m_2}{m_1+m_2}\frac{w_iw_j}{m_1} & \delta_{ij}-\frac{2m_1m_2}{m_1+m_2}\frac{w_iw_j}{m_2}
		\end{array}\right].
	\end{align*}
	Adding the 4th-6th row of $J$ to the 1st-3rd row of $J$, respectively, then subtracting the 1st-3rd column of $J$ from the 4th-6th column of $J$, respectively gives 
	\begin{align*}
		\det J &= \det \left[\begin{array}{cc} \delta_{ij} & 0 \cr
			\frac{2m_1m_2}{m_1+m_2}\frac{w_iw_j}{m_1} & \delta_{ij}-\frac{2m_1m_2}{m_1+m_2}\frac{w_iw_j}{m_2}-\frac{2m_1m_2}{m_1+m_2}\frac{w_iw_j}{m_1}
		\end{array}\right] .
	\end{align*}
	Thus we have
	\begin{align*}
		\det J &= \det \left(\delta_{ij}-\frac{2m_1m_2}{m_1+m_2}\frac{w_iw_j}{m_2}-\frac{2m_1m_2}{m_1+m_2}\frac{w_iw_j}{m_1}\right) = \det (\delta_{ij}-2w_iw_j) = -1.
	\end{align*}
\end{proof}
\begin{lemma}\label{Qcomp} For $i,j,k=1,2$, and $i\neq j$, we have
\begin{align*}
&\noindent(1) \quad \int_{\mathbb{R}^3}\phi(p)Q_{kk}(f_k,f_k) dp = \frac{1}{4}\int_{\mathbb{R}^3}\int_{\mathbb{R}^3}\int_{\mathbb{S}^2}(\phi(p)+\phi(p_*)-\phi(p')-\phi(p_*'))\cr
&\hspace{20mm}\times B_{kk}\left(\bigg|\frac{p}{m_k}-\frac{p_*}{m_k}\bigg|,w\right)
\{f_k'f_{k,*}'(1\pm f_k)(1\pm f_{k,*})-f_kf_{k,*}(1\pm f_k')(1\pm f_{k,*}')\} dwdp_*dp, \cr
&\noindent(2) \quad \int_{\mathbb{R}^3}\phi(p)Q_{ij}(f_i,f_j) dp = \frac{1}{2}\int_{\mathbb{R}^3}\int_{\mathbb{R}^3}\int_{\mathbb{S}^2}(\phi(p)-\phi(p')) B_{ij}\left(\bigg|\frac{p}{m_i}-\frac{p_*}{m_j}\bigg|,w\right) \cr
&\hspace{20mm}\times
\{f_i'f_{j,*}'(1+\tau(i) f_i)(1+\tau(j)f_{j,*})-f_if_{j,*}(1+\tau(i) f_i')(1+ \tau(j)f_{j,*}')\} dwdp_*dp.
\end{align*}
where $\tau(i)=-1 $ when $f_i$ denotes the distribution of fermion and $\tau(i)=+1 $ when $f_i$ denotes the distribution of boson.
\end{lemma}
\begin{proof}
Taking the change of variables $(p,p_*) \leftrightarrow (p_*,p)$ and $(p,p_*) \leftrightarrow (p',p_*')$, together with Lemma \ref{J}, gives (1). To prove (2), we first observe that
the collision kernel $B_{ij}$ is invariant under the change of variables $(p,p_*) \leftrightarrow (p',p_*')$ since
\begin{align*}
\bigg|\frac{p'}{m_i}-\frac{p_*'}{m_j}\bigg|&= 
\bigg|\frac{p}{m_i}-\frac{p_*}{m_j}-2w\left[\left(\frac{p}{m_i}-\frac{p_*}{m_j}\right)\cdot w\right]\bigg| = \bigg|\frac{p}{m_i}-\frac{p_*}{m_j}\bigg|.
\end{align*}
Therefore, applying the change of variables $(p,p_*) \leftrightarrow (p',p_*')$ together with Lemma \ref{J} gives the desired results. 
\end{proof}
\begin{remark}
We note that the exchange $(p,p_*) \leftrightarrow (p_*,p)$ does not leads to $(p',p_*')\leftrightarrow (p_*',p')$ in the collision opeartors $Q_{12}$ and $Q_{21}$ 
unless $m_1= m_2$. For example, if we change the notation $(p,p_*) \leftrightarrow (p_*,p)$ in $Q_{12}$,  we get
\begin{align*}
p'&= p-\frac{2m_1m_2}{m_1+m_2}w\left[\left(\frac{p}{m_1}-\frac{p_*}{m_2}\right)\cdot w\right] \quad \rightarrow \quad  p_*+\frac{2m_1m_2}{m_1+m_2}w\left[\left(\frac{p}{m_2}-\frac{p_*}{m_1}\right)\cdot w\right],
\end{align*}
which is not equal to $p_*'$ of $Q_{12}$. This is why $Q_{ij} ~(i\neq j)$ do not have the full symmetry as in $(1)$.
\end{remark}

\noindent$\bullet$ {\bf Proof of (\ref{cancellation}):} We only consider the last identity in (\ref{cancellation}), since other identities can be
treated in a similar and simpler manner. In view of the fact that the post collisional variables $(p',p_*')$ in $Q_{12}$ and $Q_{21}$ take different forms, we use the notation $\{p'\}_{12}$, $\{p_*'\}_{12}$ and $\{p'\}_{21}$, $\{p_*'\}_{21}$ to denote $p'$ and $p_*'$ in $Q_{12}$ and $Q_{21}$, respectively. We substitute $\phi(p)=|p|^2/2m_1$ in $Q_{12}$ and
use Lemma \ref{Qcomp} (2) to get
\begin{align}\label{Q12}
\begin{split}
\int_{\mathbb{R}^3}Q_{12}(f_1,f_2)\frac{|p|^2}{2m_1} dp &= \frac{1}{2}\int_{\mathbb{R}^3}\int_{\mathbb{R}^3}\int_{\mathbb{S}^2}\left(\frac{|p|^2}{2m_1}-\frac{|\{p'\}_{12}|^2}{2m_1}\right) B_{12}\left(\bigg|\frac{p}{m_1}-\frac{p_*}{m_2}\bigg|,w\right) \cr
& \quad \times \{f_1(\{p'\}_{12})f_2(\{p_*'\}_{12})(1+\tau(1) f_1(p))(1+\tau(2)f_2(p_*))\cr 
& \quad -f_1(p)f_2(p_*)(1+\tau(1) f_1(\{p'\}_{12}))(1+ \tau(2)f_2(\{p_*'\}_{12}))\} dwdp_*dp.
\end{split}
\end{align}
Similarly, substituting $\phi(p)=\frac{|p|^2}{2m_2}$ in $Q_{21}$  gives
\begin{align}\label{Q21}
\begin{split}
\int_{\mathbb{R}^3}Q_{21}(f_2,f_1)\frac{|p|^2}{2m_2} dp &= \frac{1}{2}\int_{\mathbb{R}^3}\int_{\mathbb{R}^3}\int_{\mathbb{S}^2}\left(\frac{|p|^2}{2m_2}-\frac{|\{p'\}_{21}|^2}{2m_2}\right) B_{21}\left(\bigg|\frac{p}{m_2}-\frac{p_*}{m_1}\bigg|,w\right) \cr
&\quad \times \{f_2(\{p'\}_{21})f_1(\{p_*'\}_{21})(1+\tau(2) f_2(p))(1+\tau(1)f_1(p_*)) \cr
&\quad -f_2(p)f_1(p_*)(1+\tau(2) f_2(\{p'\}_{21}))(1+ \tau(1)f_1(\{p_*'\}_{21}))\} dwdp_*dp.
\end{split}
\end{align}
We then note that the exchange of variables $(p,p_*) \leftrightarrow (p_*,p)$ in (\ref{Q21}) yields 
\begin{align*}
\{p'\}_{21} &= p-\frac{2m_2m_1}{m_2+m_1}w\left[\left(\frac{p}{m_2}-\frac{p_*}{m_1}\right)\cdot w\right]  ~ \rightarrow ~ p_*+\frac{2m_1m_2}{m_1+m_2}w\left[\left(\frac{p}{m_1}-\frac{p_*}{m_2}\right)\cdot w\right] =\{p_*'\}_{12}, \cr
\{p_*'\}_{21} &= p_*+\frac{2m_2m_1}{m_2+m_1}w\left[\left(\frac{p}{m_2}-\frac{p_*}{m_1}\right)\cdot w\right]  ~ \rightarrow ~   p-\frac{2m_1m_2}{m_1+m_2}w\left[\left(\frac{p}{m_1}-\frac{p_*}{m_2}\right)\cdot w\right] = \{p'\}_{12},
\end{align*}
so that 
\begin{align*}\label{Q12=}
\int_{\mathbb{R}^3}Q_{21}(f_2,f_1)\frac{|p|^2}{2m_2} dp &= \frac{1}{2}\int_{\mathbb{R}^3}\int_{\mathbb{R}^3}\int_{\mathbb{S}^2}\left(\frac{|p_*|^2}{2m_2}-\frac{|\{p_*'\}_{12}|^2}{2m_2}\right) B_{21}\left(\bigg|\frac{p}{m_1}-\frac{p_*}{m_2}\bigg|,w\right) \cr
&\quad \times \{f_2(\{p_*'\}_{12})f_1(\{p'\}_{12})(1+\tau(2) f_2(p_*))(1+\tau(1)f_1(p)) \cr
&\quad -f_2(p_*)f_1(p)(1+\tau(2) f_2(\{p_*'\}_{12}))(1+ \tau(1)f_1(\{p'\}_{12}))\} dwdp_*dp.
\end{align*}
Now, we combine \eqref{Q12} and (\ref{Q21}) and recall  $B_{12}=B_{21}$ to obtain 
\begin{align*}
\int_{\mathbb{R}^3}Q_{12}(f_1,f_2)&\frac{|p|^2}{2m_1} dp +\int_{\mathbb{R}^3}Q_{21}(f_2,f_1)\frac{|p|^2}{2m_2} dp \cr
&= \frac{1}{2}\int_{\mathbb{R}^3}\int_{\mathbb{R}^3}\int_{\mathbb{S}^2}\left(\frac{|p|^2}{2m_1}+\frac{|p_*|^2}{2m_2}-\frac{|\{p'\}_{12}|^2}{2m_1}-\frac{|\{p_*'\}_{12}|^2}{2m_2}\right) B_{12}\left(\bigg|\frac{p}{m_1}-\frac{p_*}{m_2}\bigg|,w\right) \cr
& \quad \times \{f_1(\{p'\}_{12})f_2(\{p_*'\}_{12})(1+\tau(1) f_1(p))(1+\tau(2)f_2(p_*))\cr 
& \quad -f_1(p)f_2(p_*)(1+\tau(1) f_1(\{p'\}_{12}))(1+ \tau(2)f_2(\{p_*'\}_{12}))\} dwdp_*dp.
\end{align*}
The r.h.s vanishes due to the microscopic energy conservation law   \eqref{mom_after_coll} with $(i,j)=(1,2)$, which gives desired result. \newline

\subsection{Conservation laws: $v$ vs $p$}
Let $\bar{f}(x,v,t)$ denote the velocity distribution function and $f(x,p,t)$ denote the momentum distribution function. Then we can reconcile the conservation laws w.r.t the velocity distribution $\bar{f}(x,v,t)$ and the conservation laws w.r.t $f(x,p,t)$ upon imposing $(i=1,2)$
\begin{align*}
	\bar{f}_i(x,v,t)=\bar{f}_i\Big(x,\frac{p}{m_i},t\Big) = m_i^3f_i(x,p,t).
\end{align*}
This relation, together with the change of variable $m_iv=p$ gives
\begin{align*}
\int \bar{f}_i(x,v,t)dxdv = \int \bar{f}_i\left(x,\frac{p}{m_i},t\right)dxdv
&=\int \frac{1}{m_i^3}\bar{f}_i\Big(x,\frac{p}{m_i},t\Big)dxdp = \int f_i(x,p,t)dxdp .
\end{align*}
Similarly, we have 
\begin{align*}
\int\bar{f}_i(x,v,t)\left(\begin{array}{c} m_iv \cr \frac{1}{2}m_1|v|^2\end{array}\right)dxdv&=\int \bar{f}_i\left(x,\frac{p}{m_i},t\right)\left(\begin{array}{c} p \cr \frac{1}{2m_i}|p|^2\end{array}\right)dxdv\cr	
&=\int\frac{1}{m^3_i}\bar{f}_i\left(x,\frac{p}{m_i},t\right)\left(\begin{array}{c} p \cr \frac{1}{2m_i}|p|^2\end{array}\right)dxdp\cr
& =\int f_i(x,p,t)\left(\begin{array}{c} p \cr \frac{1}{2m_1}|p|^2\end{array}\right)dxdp.
\end{align*}

\noindent{\bf Acknowledgement:}
Christian Klingenberg  acknowledges support by the DFG grant KL-566/20-2.
Marlies Pirner is supported by the Austrian Science Fund (FWF) project F65 and the Humboldt foundation.
Seok-Bae Yun is supported by Samsung Science and Technology Foundation under Project Number SSTF-BA1801-02.
\bibliographystyle{amsplain}

\end{document}